\documentclass[12pt,a4paper]{article}

\usepackage[leqno]{amsmath}
\usepackage{amssymb,amsthm,upref,amscd}
\usepackage[T1]{fontenc}
\usepackage{times}
\usepackage{cite}
\usepackage{amsfonts}
\usepackage[colorinlistoftodos]{todonotes}
\usepackage{setspace}

\usepackage{color}
\usepackage{hyperref}
\usepackage{graphicx}
\usepackage{wrapfig}

\usepackage[usenames,dvipsnames]{pstricks}
\usepackage{epsfig}

\newtheorem{thm}{Theorem}[section]
\newtheorem{lem}[thm]{Lemma}

\theoremstyle{definition}

\theoremstyle{remark}
\newtheorem*{notation}{Notation}
\theoremstyle{remark}
\newtheorem{remark}[thm]{Remark}
\numberwithin{equation}{section}

\newcommand{\R}{{\mathbb R}}

\definecolor{blu}{rgb}{0,0,1}

\newcommand{\al}{\alpha}
\newcommand{\be}{\beta}

\newcommand{\de}{\delta}
\newcommand{\la}{\lambda}

\newcommand{\De}{\Delta}

\newcommand{\vphi}{\varphi}
\newcommand{\eps}{\varepsilon}
\newcommand{\pa}{\partial}

\newcommand{\beq[1]}{\begin{equation}\label{eq:#1}}
\newcommand{\eeq}{\end{equation}}

\begin{document}
\title{Existence and orbital stability of standing waves for nonlinear Schr\"odinger systems}

\author{\sc{Tianxiang Gou}
\and
\sc{Louis Jeanjean}}

\date{}
\maketitle

\begin{abstract}
In this paper we investigate the existence of solutions in $H^1(\R^N) \times H^1(\R^N)$ for nonlinear Schr\"odinger systems of the form
\[
\left\{
\begin{aligned}
-\De u_1 &= \la_1 u_1 + \mu_1 |u_1|^{p_1 -2}u_1
            + r_1\be |u_1|^{r_1-2}u_1|u_2|^{r_2}, \\
-\De u_2 &= \la_2 u_2 + \mu_2 |u_2|^{p_2 -2}u_2
            + r_2 \be |u_1|^{r_1}|u_2|^{r_2 -2}u_2,
\end{aligned}
\right.
\]
under the constraints
\[
\int_{\R^N}|u_1|^2 \, dx = a_1>0,\quad \int_{\R^N}|u_2|^2 \, dx = a_2>0.
\]
Here $ N \geq 1, \beta >0, \mu_i >0, r_i >1,  2 <p_i < 2 + \frac{4}{N}$ for $i=1,2$ and $ r_1 + r_2 < 2 + \frac{4}{N}$. This problem is motivated by the search of standing waves for an evolution problem appearing in several physical models. Our solutions are obtained as constrained global minimizers of an associated functional. Note that in the system  $\la_1$ and $\la_2$ are unknown and will correspond to the Lagrange multipliers. Our main result is the precompactness of the minimizing sequences, up to translation. Assuming the local well posedness of the associated evolution problem we then obtain the orbital stability of the standing waves associated to the set of minimizers.
\end{abstract}



\section{Introduction}\label{sec:intro}

We consider the existence of solutions to a nonlinear Schr\"odinger system of the form
\beq[1.1]
\left\{
\begin{aligned}
-\De u_1 &= \la_1 u_1 + \mu_1 |u_1|^{p_1 -2}u_1
            + r_1\be |u_1|^{r_1-2}u_1|u_2|^{r_2}, \\
-\De u_2 &= \la_2 u_2 + \mu_2 |u_2|^{p_2 -2}u_2
            + r_2 \be |u_1|^{r_1}|u_2|^{r_2 -2}u_2,
\end{aligned}
\right.
\eeq
satisfying the conditions
\beq[1.2]
\int_{\R^N} |u_1|^2 \, dx = a_1, \quad \int_{\R^N}| u_2|^2 \, dx = a_2.
\eeq
Here $a_1, a_2>0$ are prescribed and we shall assume throughout the paper

\begin{itemize}
\item[(H0)]  $ N \geq 1, \beta >0, \mu_i >0, r_i >1,  2 <p_i < 2 + \frac{4}{N}$ for $i=1,2$ and $ r_1 + r_2 < 2 + \frac{4}{N}$.
\end{itemize}

The problem under consideration is associated to the research of standing waves, namely, solutions having the form
\[
\Psi_1(t,x) = e^{-i\la_1 t} u_1(x), \quad \Psi_2(t,x) = e^{-i\la_2 t} u_2(x),
\]
for some $\la_1, \la_2 \in \R$, of the nonlinear Schr\"odinger system
\begin{equation}\label{1.3}
\begin{cases}
- i \pa_t \Psi_1 = \De \Psi_1 + \mu_1 |\Psi_1|^{p_1-2}\Psi_1
                   +\be |\Psi_1|^{r_1-2}\Psi_1|\Psi_2|^{r_2},\\
- i \pa_t \Psi_2 = \De \Psi_2 + \mu_2 |\Psi_2|^{p_2-2}\Psi_2
                   +\be |\Psi_1|^{r_1}|\Psi_2|^{r_2-2}\Psi_2,
\end{cases} \text{in $\R \times \R^N$}.
\end{equation}
This system comes from mean field models for binary mixtures of Bose-Einstein condensates or for binary gases of fermion atoms in degenerate quantum states (Bose-Fermi mixtures, Fermi-Fermi mixtures), see \cite{Bagnato-etal, EG, M}.

One motivation to look for normalized solutions of system \eqref{eq:1.1} is that the masses
\[
\int_{\R^N} |\Psi_1|^2 \, dx \quad \text{and} \quad \int_{\R^N} |\Psi_2|^2 \, dx
\]
are preserved along the trajectories of \eqref{1.3}. Our solutions of \eqref{eq:1.1}-\eqref{eq:1.2} will be obtained as minimizers of the functional
\[
J(u_1,u_2)
 := \frac12\int_{\R^N} |\nabla u_1|^2 + |\nabla u_2|^2 \, dx
     - \int_{\R^N} \frac{\mu_1}{p_1}|u_1|^{p_1}
     + \frac{\mu_2}{p_2}  |u_2|^{p_2}
     + \be |u_1|^{r_1}|u_2|^{r_2} \, dx
\]
constrained on
$$
S(a_1, a_2): = \{(u_1, u_2)\in H^1(\R^N) \times H^1(\R^N): \|u_1\|_2^2 = a_1, \|u_2\|_2^2 = a_2\}.
$$
Namely we are to consider the  minimization problem
\beq[1.4]
m(a_1, a_2):=\inf_{(u_1, u_2)\in S(a_1, a_2)} J(u_1, u_2).
\eeq
It is standard that the minimizers of \eqref{eq:1.4} are solutions to \eqref{eq:1.1}-\eqref{eq:1.2} where $\la_1, \la_2$ appear as the Lagrange multipliers. Actually the existence of minimizers for \eqref{eq:1.4} will be obtained as a consequence of the stronger statement that any minimizing sequence for \eqref{eq:1.4} is, up to translation, precompact.

\begin{thm} \label{th:min}
Assume (H0). Then for any $a_1 >0$ and $a_2 >0$ all minimizing sequences for \eqref{eq:1.4} are precompact in $H^1(\R^N)\times H^1(\R^N)$ after a suitable translation.
\end{thm}


Following some initial works \cite{St1,St2}, the compactness concentration principle of P.L. Lions \cite{Li1, Li2} has had, over the last thirty years, a deep influence on solving minimization problems under constraints.
Heuristic arguments readily convince that in our problem the compactness of any minimizing sequence holds if the following strict subadditivity conditions are satisfied.
\beq[1.5]
m(a_1, a_2) < m(b_1, b_2)+ m(a_1-b_1, a_2-b_2),
\eeq
where $ 0\leq b_i<a_i $ for $ i=1, 2 $ and $ b_1 + b_2 \neq 0 $.

To deal with just one constraint, several techniques have been developed to prove strict subadditivity conditions. Most are based on some homogeneity type property. In autonomous case it is also possible to use scaling arguments, see for example \cite{BS,CJS,Sh1}. In the case of multiple constraints how to establish strict subadditivity conditions is much less understood. As a matter of fact few papers address the issue of compactness of minimizing sequences for systems as \eqref{eq:1.1}-\eqref{eq:1.2}. Moreover in most of them there is either exactly one constraint \cite{CiZu} or the two constraints cannot be chosen independently \cite{NW2,NW1,Oh}. Concerning \eqref{eq:1.4} the more complete results seem to be due to \cite{NW}. In \cite{NW} the precompactness of minimizing sequences is obtained assuming $N=1$. To exclude the dichotomy the authors crucially applied \cite[Lemma 2.10]{AB} which depends in turn on original ideas introduced in \cite{By}, see also \cite{Ga}. In \cite[Lemma 2.10]{AB} it is shown that the $H^1(\R)$ norm of some functions are strictly decreasing when the masses of the functions are symmetrically rearranged. See also \cite{LiNgWa2} for similar arguments on related problems.


If one is merely interested in the existence of one minimizer, two papers should be mentioned. In \cite{CaChWe} the existence of one minimizer had been achieved still for $N=1$. The restriction on the dimension was subsequently removed in \cite{BJ} where the existence of a minimizer for \eqref{eq:1.4} was obtained in full generality in $H^1(\R^N)$ for $N=2,3,4$ and under some restrictions for $N \geq 5$, see \cite[Theorem 2.1]{BJ} for a precise statement.

In this paper, inspired by \cite{Ik}, we propose an alternatively simple approach to verify the compactness of the minimizing sequences for \eqref{eq:1.4} in any dimension. It is standard that any minimizing sequence $\{(u_1^n,u_2^n)\} \subset S(a_1,a_2)$ is bounded in $H^1(\R^N)\times H^1(\R^N)$ and thus without restriction we can assume that $u_1^n \rightharpoonup u_1$ and $u_2^n \rightharpoonup u_2$ weakly in $H^1(\R^N)$. To demonstrate the strong convergence in $H^1(\R^N) \times H^1(\R^N)$, we first prove the weaker result that, up to translation, $\{(u_1^n, u_2^n)\}$ is strongly convergent in $L^p(\R^N) \times L^p(\R^N)$, $2<p<2^*$. Because we deal with a minimizing sequence it is clear that neither $\{u_1^n\}$ nor $\{u_2^n\}$  can vanish.
We also observe that, in contrast to what happens in $L^2(\R^N)$, the non compactness of $\{u_i^n\}$ in $L^p(\R^N)$ implies the existence of at least two bumps going apart one from another. By bumps we mean here exist a $R < \infty$ and a sequence $\{y_n\} \subset \R^N$ such that $\liminf_{n \to \infty} \int_{B(y_n, R)}|u_i^n|^p dx >0.$ At this point we make use of a very nice result of M. Shibata \cite{Sh2} as presented in \cite[Lemma A.1]{Ik}. This result, which can somehow be considered as an extension of \cite[Lemma 2.10]{AB} to any  dimension, shows that the existence of two or more bumps for one of the sequence $\{u_i^n\}$ contradicts with its  minimizing character. At this point we have proved the compactness of each $\{u_i^n\}$ in $L^p(\R^N)$ and we end the proof of the convergence in $L^p(\R^N)\times L^p(\R^N)$ by showing that the bumps of $\{u_1^n\}$ and $\{u_2^n\}$ cannot move away from each other.

With this convergence,
assuming that $(u_1, u_2)$ is the weak limit of one minimizing sequence $\{(u_1^n, u_2^n)\}$, we have that $J(u_1, u_2) \leq m(a_1, a_2)$. Namely our functional is lower semicontinuous on minimizing sequences.  If $\|u_1\|_2^2 = a_1$ and $\|u_2\|_2^2 = a_2$ the strong convergence in $H^1(\R^N) \times H^1(\R^N)$ immediately results. Suppose not and assume that $\|u_1\|_2^2 := b_1 < a_1$ or $\|u_2\|_2^2 := b_2 < a_2$. Since $J(u_1, u_2) \leq m(a_1, a_2)$ it follows that $m(b_1,b_2) \leq m(a_1,a_2)$. We then reach a contradiction via observing that the weak version \eqref{eq:1.5}, where an equality is allowed, always holds and that it implies that the function $(c_1,c_2) \mapsto m(c_1,c_2)$ is strictly decreasing in both arguments. For related observations we refer to \cite{JS}, see also \cite{Ik}.

\begin{remark}\label{rem:oneminimizer}
Note that if one is just interested in the existence of one minimizer for \eqref{eq:1.4} a shorter proof can be given. Choosing a minimizing sequence $\{(u_1^n,u_2^n)\} \subset S(a_1,a_2)$  which consists of Schwarz symmetric functions then, thanks to the compact embedding of $H_r^1(\R^N)$ into $L^p(\R^N)$, $2<p<2^*$ (here $H_r^1(\R^N)$ denotes the the subspace of radially symmetric functions of $H^1(\R^N)$), it readily follows that if $(u_1,u_2)$ is the weak limit of $\{(u_1^n,u_2^n)\}$ then $J(u_1,u_2) \leq m(a_1,a_2)$. The rest of the proof is identical to the one of Theorem \ref{th:min}. Alternately it is possible to obtain the existence of a minimizer working directly in $H_r^1(\R^N)\times H_r^1(\R^N)$. In that direction we refer to Remark \ref{rem:radial} later in this paper.
\end{remark}



\begin{remark}
The scheme to treat the compactness of minimizing sequences for \eqref{eq:1.4} could be carried to deal with $n$ constraints minimization problems on $\R^N$. More precisely,
$$
m(a_1, \cdots, a_n):= \inf_{S(a_1, \cdots, a_n)} J(u_1, \cdots, u_n),
$$
where $S(a_1, \cdots, a_n) := \{(u_1, \cdots, u_n) \in H^1(\R^N) \times \cdots \times H^1(\R^N): \|u_i\|_2^2 = a_i >0
\,\, \text{ for}  \,\, i=1, \cdots, n\}$,
$$
J(u_1, \cdots, u_n) := \frac12 \int_{\R^N} \sum_{i=1}^n |\nabla u_i|^2 \, dx
                    - \int_{\R^N} \sum_{i=1}^n {\frac{\mu_i}{p_i}} |u_i|^{p_i}
                     + \sum_{i \neq j}^n \be_{ij} |u_i|^{r_i} |u_j|^{r_j} \, dx,
$$
and $N \geq 1, \mu_i > 0, \be_{ij} > 0$,
$2 < p_i < 2+ \frac4N, r_i, r_j >1, r_i+r_j < 2+ \frac4N$ for $i, j = 1, \cdots, n.$
\end{remark}

Set
\begin{align*}
G(a_1,a_2) := \{(u_1, u_2) \in S(a_1, a_2): J(u_1, u_2) = m(a_1, a_2)\}.
\end{align*}

%
%

Note that under  assumption (H0) it is not known if \eqref{1.3} is locally well posed. The point being that when $1 <r_i < 2$ for  $i=1,2$ the interaction part is not Lipchitz continuous and in particular the uniqueness may fail. As a consequence our last result which states the orbital stability  of the set of standing waves associated to $G(a_1,a_2)$ is only valid under condition.

\begin{thm}\label{th:sta}
Assume (H0) and the local existence of the Cauchy problem in \eqref{1.3}. Then the set $G(a_1, a_2)$ is orbitally stable, i.e. for any $\eps>0$, there exists $\de>0$
so that if the initial condition $(\psi_1(0), \psi_2(0))$ in system \eqref {1.3} satisfies
\begin{equation*}
\inf_{(u_1, u_2)\in G(a_1, a_2)}
\|(\psi_1(0), \psi_2(0)) - (u_1,u_2)\|\leq \de,
\end{equation*}
then
\begin{equation*}
\sup_{t \geq 0} \inf_{(u_1, u_2) \in G(a_1, a_2)}
\|(\psi_1(t), \psi_2(t)) - (u_1, u_2)\| \leq \eps,
\end{equation*}
where $(\psi_1(t), \psi_2(t))$ is the solution of system \eqref{1.3}
corresponding to the initial condition $(\psi_1(0), \psi_2(0))$ and $\|\cdot\|$ denotes the norm in Sobolev space $H^1(\R^N)$.
\end{thm}

This paper is organized as follows: In Section \ref{sec:pre}, we display some preliminary results.
Theorem \ref{th:min} will be completed in Section \ref{sec:pf-th:min}. Section \ref{sec:pf-th:sta} is devoted to Theorem \ref{th:sta}. \medskip

\textbf{Acknowledgements.} The authors  thank N. Ikoma for pointing to them that the local well posedness of the Cauchy problem was not known under the assumptions of Theorem \ref{th:min}. They also thank T. Luo for useful observations on a preliminary version. Finally note that this work has been carried out in the framework of the Project NONLOCAL (ANR-14-CE25-0013), funded by the French National Research Agency (ANR).\\

\begin{notation}
In this paper it is understood that all functions, unless otherwise stated, are complex-valued, but for simplicity
we write $L^p(\R^N),H^1(\R^N)...,$ for any $1\leq p<\infty,$ $L^p(\R^N)$ is the usual Lebesgue space with norm
$$\|u\|_p^p := \int_{\R^N}|u|^p\,dx,$$
and $H^1(\R^N)$ the usual Sobolev space endowed with the norm
$$\|u\|^2 := \int_{\R^N}|\nabla u|^2+|u|^2 \, dx.$$
We denote by $'\rightarrow'$ and $'\rightharpoonup'$ strong convergence and weak convergence, respectively, in corresponding space, and denote by $B(x, R)$ a ball in $\R^N$ of center $x$ and radius $R>0.$
\end{notation}

\section{Preliminary results}\label{sec:pre}

Firstly, let us observe that the functional $J$ is well defined in $H^1(\R^N) \times H^1(\R^N)$. For $r_1, r_2 > 1, r_1+r_2 < 2+\frac4N,$ there is $q>1$
with $2< r_1 q, r_2 q' \leq 2^*, q':= \frac{q}{q-1}$. Hence
\[
\int_{\R^N} |u_1|^{r_1}|u_2|^{r_2} \, dx
   \le \|u_1\|_{r_1q}^{r_1}\|u_2\|_{r_2q'}^{r_2}
   < \infty.
\]
The Gagliardo-Nirenberg inequality
\[
\|u\|_p \le C(N,p)\|\nabla u\|_2^\al \|u\|_2^{1-\al}\quad
 \text{where } \al=\frac{N(p-2)}{2p},
\]
which holds for $u \in H^1(\R^N)$ and $2 \le p \le 2^*$, implies for $(u_1, u_2) \in S(a_1, a_2)$:
\begin{align} \label{eq:2.1}
\begin{split}
&\int_{\R^N} |u_1|^{p_1} \, dx
     \le C(N,p_1,a_1) \|\nabla u_1\|_2^{\frac{N(p_1-2)}{2}}, \\
&\int_{\R^N} |u_2|^{p_2} \, dx
     \le C(N,p_2,a_2) \|\nabla u_2\|_2^{\frac{N(p_2-2)}{2}},
\end{split}
\end{align}
and
\beq[2.2]
\int_{\R^N} |u_1|^{r_1}|u_2|^{r_2} \, dx
  \le \|u_1\|_{r_1q}^{r_1}\|u_2\|_{r_2q'}^{r_2}
  \le C\|\nabla u_1\|_2^{\frac{N(r_1q-2)}{2q}}
  \|\nabla u_2\|_2^{\frac{N(r_2q'-2)}{2q'}}
\eeq
with $C=C(N,r_1,r_2,a_1,a_2,q)$. \medskip

Now recall the rearrangement results of Shibata \cite{Sh2} as presented in \cite{Ik}. Let $u$ be a Borel measurable function on $\R^N$. It is said to vanish at infinity if $|\{x \in \R^N: |u(x)|>t\}|< \infty$ for every $t>0$.
Here $|A|$ stands for the $N$-dimensional Lebesgue measure of a Lebesgue mesurable  set $A \subset \R^N$.
Considering two Borel mesurable functions $u,v$ which vanish at infinity in $\R^N$, we define for $t>0$,
$A^{\star}(u, v; t):= \{x \in \R^N: |x|<r\}$ where $r>0$ is chosen so that
$$
|B(0, r)| = |\{x \in \R^N: |u(x)|>t\}| + |\{x \in \R^N: |v(x)|>t\}|,
$$
and $\{u, v\}^{\star}$ by
$$
\{u, v\}^{\star}(x):= \int_{0}^{\infty} \chi_{A^{\star}(u, v; t)} (x) \, dt,
$$
where $\chi_A(x)$ is a characteristic function of the set $A \subset \R^N$.

\begin{lem} \cite [Lemma A.1]{Ik} \label{lem:relem}
\begin{itemize}
\item[(i)] The function $\{u, v\}^{\star}$ is radially symmetric, non-increasing and lower semi-continuous.
           Moreover, for each $t > 0$ there holds $\{ x \in \R^N : \{u, v\}^{\star} > t\} = A^{\star}(u, v; t)$.
\item[(ii)] Let $\Phi : [0, \infty) \rightarrow [0, \infty)$ be non-decreasing, lower semi-continuous, continuous at $0$
           and $\Phi(0) = 0$. Then $\{\Phi(u),\Phi(v)\}^{\star} = \Phi(\{u, v\}^{\star})$.
\item[(iii)] $\|\{u, v\}^{\star}\|_p^p = \|u\|_p^p + \|v\|_p^p$ \, for $ 1 \leq p < \infty$.
\item[(iv)] If $u, v \in H^1(\R^N)$, then $\{u, v\}^{\star} \in  H^1(\R^N)$
            and $\|\nabla \{u ,v\}^{\star}\|_2^2 \leq \|\nabla u\|_2^2 + \|\nabla v\|_2^2$.
            In addition, if $u, v \in (H^1(\R^N) \cap C^1(\R^N)) \setminus \{0\}$ are radially symmetric, positive and non-increasing, then
            $$
            \int_{\R^N} |\nabla\{u ,v\}^{\star}|^2 \, dx < \int_{\R^N} |\nabla u|^2 + \int_{\R^N} |\nabla v|^2\, dx.
            $$
\item[(v)] Let $u_1, u_2, v_1, v_2 \geq 0$ be Borel measurable functions which vanish at infinity, then
            $$
            \int_{\R^N} (u_1u_2 + v_1v_2) \,dx \leq \int_{\R^N} \{u_1, v_1\}^{\star}\{u_2, v_2\}^{\star} \, dx.
            $$
\end{itemize}
\end{lem}

\section{Proof of Theorem \ref{th:min}}\label{sec:pf-th:min}

Hereafter, we use the same notation $m(a_1, a_2)$ for $a_1, a_2\geq 0$, namely, one component of $(a_1, a_2)$ may be zero.

In what follows, we collect some basic properties of $m(a_1,a_2).$

\begin{lem} \label{lem:m}
\begin{itemize}
\item[(i)] For any  $a_1, a_2\geq 0$ with either $a_1 >0$ or $a_2 >0$,
 $$-\infty < m(a_1, a_2) < 0.$$
\item[(ii)] $m(a_1, a_2)$ is continuous with respect to $a_1, a_2\geq 0$.
\item[(iii)] For any $a_1 \geq b_1 \geq 0, a_2 \geq b_2 \geq 0, m(a_1, a_2) \leq m(b_1, b_2) + m(a_1-b_1, a_2-b_2)$.
\end{itemize}
\end{lem}

\begin{proof}
(i) Observe that $\frac{N(p_i-2)}{2} < 2$ by $p_i < 2+ \frac4N$ for $i = 1,2$ and that
$$\frac{N(r_1q-2)}{2q} + \frac{N(r_2q'-2)}{2q'} < 2,$$
owing to $r_1+r_2<2+\frac4N$. Thus, it follows from \eqref {eq:2.1}-\eqref{eq:2.2} that $J$ is coercive and in particular
$m(a_1, a_2)>-\infty.$ Now taking into account that $\beta > 0$, one has
\[
m(a_1, a_2) \leq m(a_1,0) + m(0,a_2).
\]
Since $2<p_i < 2 + \frac{4}{N}$ for $i=1,2$, it is standard to show that $m(a_1,0) <0$ (if $a_1>0$) and $m(0,a_2) <0$ (if $a_2 >0$). Thus
$m(a_1,a_2) <0.$

(ii) We assume $(a_1^n, a_2^n)=(a_1, a_2)+o(1)$. From the definition of $m(a_1^n ,a_2^n)$,
for any $\eps>0$, there exists $(u_1^n, u_2^n) \in S(a_1^n, a_2^n)$ such that
\beq[2.3]
J(u_1^n, u_2^n)\leq m(a_1^n, a_2^n)+ \eps.
\eeq
Setting
$$v_i^n := \frac{u_i^n}{\|u_i^n\|_2}a_i^{\frac12}$$
for $i = 1 ,2,$ we have that $(v_1^n, v_2^n) \in S(a_1, a_2)$ and
\beq[2.4]
m(a_1, a_2)\leq J(v_1^n, v_2^n) = J(u_1^n, u_2^n)+o(1).
\eeq
Combining \eqref{eq:2.3} and \eqref{eq:2.4} we obtain
$$
m(a_1, a_2) \leq m(a_1^n, a_2^n)+\eps +o(1).
$$
Reversing the argument we obtain similarly that 
$$
m(a_1^n, a_2^n) \leq m(a_1, a_2) + \eps + o(1).
$$
Therefore, since $\eps >0$ is arbitrary, we deduce that $m(a_1^n, a_2^n) = m(a_1, a_2) + o(1).$

(iii) By density of $C^{\infty}_0(\R^N)$ into $H^1(\R^N)$,
for any $\eps>0$, there exist $(\bar{\vphi}_1, \bar{\vphi}_2),
(\hat{\vphi}_1, \hat{\vphi}_2) \in C^{\infty}_0(\R^N) \times C^{\infty}_0(\R^N)$ with
$\|\bar{\vphi}_i\|_2^2 = b_i, \|\hat{\vphi}_i\|_2^2 = a_i-b_i$ for $i = 1, 2$ such that
\begin{align*}
&J(\bar{\vphi}_1, \bar{\vphi}_2) \leq m(b_1, b_2) + \frac{\eps}{2}, \\
&J(\hat{\vphi}_1, \hat{\vphi}_2) \leq m(a_1-b_1, a_2-b_2) + \frac{\eps}{2}.
\end{align*}
Since $J$ is invariant by translation, without loss of generality, we may assume that
$supp \,\bar{\vphi}_i \cap supp \,\hat{\vphi}_i =\emptyset,$  and then $\|\bar{\vphi}_i+\hat{\vphi}_i\|_2^2
= \|\bar{\vphi}_i\|_2^2+\|\hat{\vphi}_i\|_2^2 = a_i$ for $i=1,2$, as well as
$$
m(a_1, a_2) \leq
      J(\bar{\vphi}_1 + \hat{\vphi}_1, \bar{\vphi}_2 + \hat{\vphi}_2)
      \leq m(b_1, b_2) + m(a_1-b_1 ,a_2-b_2) + \eps.
$$
Thus
$$m(a_1, a_2) \leq m(b_1, b_2) + m(a_1-b_1 ,a_2-b_2).$$
\end{proof}

\begin{lem} \label{lem:Leib}
Assume $r_1, r_2 > 1, r_1 + r_2 < 2 + \frac4N.$ If $(u_1^n, u_2^n)\rightharpoonup (u_1, u_2)$ in $H^1(\R^N)\times H^1(\R^N),$
then
$$
\int_{\R^N}|u_1^n|^{r_1}|u_2^n|^{r_2} - |u_1^n-u_1|^{r_1}|u_2^n-u_2|^{r_2} \, dx
= \int_{\R^N}|u_1|^{r_1}|u_2|^{r_2}\,dx+o(1).
$$
\end{lem}
\begin{proof}
Since the lemma can be proved following closely the approach of \cite[Lemma 2.3]{CZ}, we only provide the outline of the proof.
For any $b_1, b_2, c_1, c_2 \in \R$ and $\eps>0$, set $r:=r_1 + r_2.$ The mean value theorem and Young's inequality lead
to
\begin{align*}
&\bigl||b_1 + b_2|^{r_1} |c_1 + c_2|^{r_2} - |b_1|^{r_1} |c_1|^{r_2} \bigr| \\
& \leq C \eps \bigl( |b_1|^{r} + |c_1|^{r} + |b_2|^{r} + |c_2|^{r} \bigr)
  + C_{\eps} \bigl( |b_2|^{r} + |c_2|^{r} \bigr).
\end{align*}
Denote $b_1 := u_1^n - u_1, c_1 := u_2^n - u_2, b_2 := u_1, c_2 := u_2.$ Then
\begin{align*}
f_n^{\eps} &:= \bigl[ \bigl ||u_1^n|^{r_1}|u_2^n|^{r_2} - |u_1^n-u_1|^{r_1}|u_2^n-u_2|^{r_2}
               - |u_1|^{r_1}|u_2|^{r_2} \bigr| \bigr.\\
           &- C \eps \bigl. \left( |u_1^n - u_1|^{r} +  |u_2^n - u_2|^{r} + |u_1|^{r} + |u_2|^{r} \right) \bigr]^{+} \\
           &\leq |u_1|^{r_1} |u_2|^{r_2} + C_{\eps} \left( |u_1|^{r} + |u_2|^{r} \right),
\end{align*}
where $u^{+}(x):= \max\{u(x), 0\}$, so the dominated convergence theorem implies that
\beq [domi]
\int_{\R^N} f_n^{\eps} \, dx \rightarrow 0 \quad \text{as} \,\, n \rightarrow \infty.
\eeq
Since
\begin{align*}
&\bigl ||u_1^n|^{r_1}|u_2^n|^{r_2} - |u_1^n-u_1|^{r_1}|u_2^n-u_2|^{r_2} - |u_1|^{r_1}|u_2|^{r_2} \bigr| \\
& \leq f_n^{\eps} + C \eps \bigl( |u_1^n - u_1|^{r} +  |u_2^n - u_2|^{r} + |u_1|^{r} + |u_2|^{r} \bigr),
\end{align*}
by the boundedness of $\{(u_1^n, u_2^n)\}$ in $H^1(\R^N)\times H^1(\R^N)$ and \eqref{eq:domi}, it follows that
$$
\int_{\R^N}|u_1^n|^{r_1}|u_2^n|^{r_2} - |u_1^n-u_1|^{r_1}|u_2^n-u_2|^{r_2} \, dx
= \int_{\R^N}|u_1|^{r_1}|u_2|^{r_2}\,dx+o(1).
$$
\end{proof}

\begin{lem}\label{lem:conv}
Any minimizing sequence for \eqref{eq:1.4} is, up to translation, strongly convergent in $L^p(\R^N)\times L^p(\R^N)$ for  $2<p<2^*$.
\end{lem}
\begin{proof}
Assume that $\{(u_1^n, u_2^n)\}$ is a minimizing sequence associated to the functional $J$ on $S(a_1, a_2)$. By the coerciveness of $J$ on $S(a_1,a_2)$, the sequence $\{(u_1^n, u_2^n)\}$ is bounded in $H^1(\R^N) \times H^1(\R^N)$.
If
$$
\sup_{y\in \R^N}\int_{B(y, R)}|u_1^n|^2+|u_2^n|^2\, dx = o(1),
$$
for some $R>0$, then $u_i \rightarrow 0$ in $L^p(\R^N) $ for $2 < p < 2^*, i = 1,2$, see \cite [Lemma I. 1] {Li2}. This is incompatible with the fact that $m(a_1, a_2) < 0$, see Lemma \ref{lem:m} (i).
Thus, there exist a $\beta_0 >0$ and a sequence $\{y_n\} \subset \R^N$ such that
$$
\int_{B(y_n, R)}|u_1^n|^2+|u_2^n|^2\, dx \geq \be_0 > 0,
$$
and we deduce from the weak convergence in $H^1(\R^N)\times H^1(\R^N)$ and the local compactness in $L^2(\R^N)\times L^2(\R^N)$ that $\left(u_1^n(x-y_n), u_2^n(x-y_n)\right) \rightharpoonup (u_1, u_2) \neq (0, 0)$ in $H^1(\R^N) \times H^1(\R^N)$. Our aim is to prove that $w_i^n(x) := u_i^n(x) - u_i(x+y_n) \rightarrow 0$ in $L^p(\R^N)$ for $2 < p < 2^*, i = 1, 2$ and so we suppose by contradiction that there exists a $2<q<2^*$ such that $(w_1^n, w_2^n) \nrightarrow(0, 0)$ in $L^q(\R^N) \times L^q(\R^N).$ Note that under this assumption there exists a sequence
$\{z_n\} \subset \R^N$ such that
$$
\left(w_1^n(x-z_n), w_2^n(x-z_n)\right)\rightharpoonup (w_1, w_2)\neq(0, 0)
$$
in $H^1(\R^n)\times H^1(\R^N).$ Indeed otherwise
$$
\sup_{y \in \R^N}\int_{B(y, R)}|w_1^n|^2+|w_2^n|^2\, dx = o(1),
$$
which leads to $(w_1^n, w_2^n) \rightarrow (0, 0)$ in $L^p(\R^N) \times L^p(\R^N)$ for $2<p<2^*$.

Now, combining the Brezis-Lieb Lemma, 
Lemma \ref{lem:Leib} and the translational invariance we conclude
\begin{equation}\label{3.14}
\begin{split}
&J(u_1^n, u_2^n) = J(u_1^n(x-y_n), u_2^n(x-y_n)) \\
&= J(u_1^n(x-y_n)-u_1+u_1, u_2^n(x-y_n)-u_2+u_2)\\
&= J(u_1^n(x-y_n)-u_1, u_2^n(x-y_n) - u_2) + J(u_1, u_2) + o(1) \\
&= J(w_1^n(x-y_n), w_2^n(x-y_n)) + J(u_1, u_2) + o(1) \\
&= J(w_1^n(x-z_n), w_2^n(x-z_n)) + J(u_1, u_2) + o(1) \\
&= J(w_1^n(x-z_n)-w_1+w_1, w_2^n(x-z_n)-w_2+w_2) + J(u_1, u_2) + o(1) \\
&= J(w_1^n(x-z_n)-w_1, w_2^n(x-z_n)-w_2) + J(w_1, w_2) + J(u_1, u_2) + o(1),
\end{split}
\end{equation}
and
\begin{align*}
\|u_i^n(x-y_n)\|_2^2 &= \|u_i^n(x-y_n)-u_i+u_i\|_2^2\\
                         &= \|u_i^n(x-y_n)-u_i\|_2^2 + \|u_i\|_2^2 + o(1)\\
                         &= \|w_i^n(x-z_n)-w_i+w_i\|_2^2 + \|u_i\|_2^2 + o(1)\\
                         &= \|w_i^n(x-z_n)-w_i\|_2^2 + \|w_i\|_2^2 + \|u_i\|_2^2 + o(1).
\end{align*}
Thus
\begin{equation}\label{3.15}
\begin{split}
\|w_i^n(x-z_n)-w_i\|_2^2 &= \|u_i^n(x-y_n)\|_2^2 - \|w_i\|_2^2 - \|u_i\|_2^2 + o(1)\\
                             &= a_i - \|w_i\|_2^2 - \|u_i\|_2^2 + o(1)\\
                             &=b_i + o(1),
\end{split}
\end{equation}
where $b_i := a_i - \|w_i\|_2^2 - \|u_i\|_2^2.$
Noting that
\begin{align*}
\|w_i\|_2^2 & \leq \liminf_{n \rightarrow \infty} \|w_i^n(x-z_n)\|_2^2
                         = \liminf_{n \rightarrow \infty}\|u_i(x-y_n) - u_i\|_2^2\\
                &= a_i- \|u_i\|_2^2,
\end{align*}
then $b_i \geq 0$ for $i = 1, 2.$ Recording that $J(u_1^n,u_2^n) \to m(a_1,a_2)$, in view of \eqref{3.15}, Lemma \ref{lem:m} $(ii)$ and \eqref{3.14}, we get
\beq[3.16]
m(a_1, a_2) \geq m(b_1, b_2) + J(w_1, w_2) + J(u_1, u_2).
\eeq
If $J(w_1, w_2) > m(\|w_1\|_2^2, \|w_2\|_2^2)$ or $J(u_1, u_2) > m(\|u_1\|_2^2, \|u_1\|_2^2),$ then,
from \eqref{eq:3.16} and Lemma \ref{lem:m} $(iii)$, it follows
$$
m(a_1, a_2) > m(b_1, b_2) +  m(\|w_1\|_2^2, \|w_2\|_2^2) + m(\|u_1\|_2^2, \|u_1\|_2^2) \geq m(a_1 ,a_2),
$$
which is impossible. Hence $J(w_1, w_2) = m(\|w_1\|_2^2, \|w_2\|_2^2)$ and $J(u_1, u_2) = m(\|u_1\|_2^2, \|u_1\|_2^2).$
We denote by $\tilde{u}_i, \tilde{w}_i$ the classical Schwarz symmetric-decreasing rearrangement of $u_i, w_i$ for $i = 1, 2,$. Since
$$
 \|\tilde{u}_i\|_2^2 = \|u_i\|_2^2 ,\quad
\|\tilde{w}_i\|_2^2 = \|w_i\|_2^2,
$$
$$
J(\tilde{u}_1, \tilde{u}_2) \leq J(u_1, u_2), \quad
J(\tilde{w}_1, \tilde{w}_2) \leq J(w_1, w_2)
$$
see for example \cite{LL}, we deduce that
$$J(\tilde{u}_1, \tilde{u}_2) = m(\|u_1\|_2^2, \|u_2\|_2^2), \quad J(\tilde{w}_1, \tilde{w}_2) = m(\|w_1\|_2^2, \|w_2\|_2^2).
$$
Therefore, $(\tilde{u}_1, \tilde{u}_2), (\tilde{w}_1, \tilde{w}_2)$ are solutions of the system \eqref{eq:1.1} and from standard regularity results we have that
$\tilde{u}_i, \tilde{w}_i \in C^2(\R^N)$ for $i = 1, 2.$

At this point Lemma \ref{lem:relem} comes into play. Without restriction we may assume  $u_1 \neq 0.$ We divide into two cases.

\noindent{\it Case 1: $u_1 \neq 0$ and $w_1 \neq 0$.} \\
By virtue of Lemma \ref{lem:relem} $ (ii), (iv), (v)$,
$$
\int_{\R^N} | {\nabla \{ \tilde{u}_1,  \tilde{w}_1\}}^{\star}| \, dx
         < \int_{\R^N} |\nabla \tilde{u}_1|^2 + |\nabla \tilde{w}_1|^2 \, dx
         \leq \int_{\R^N} |\nabla u_1|^2 + |\nabla w_1|^2 \, dx,
$$
\begin{align*}
\int_{\R^N}|\{\tilde{u}_1, \tilde{w}_1\}^{\star}|^{r_1}|\{\tilde{u}_2, \tilde{w}_2\}^{\star}|^{r_2} \, dx
        &= \int_{\R^N} \{|\tilde{u}_1|^{r_1}, |\tilde{w}_1|^{r_1}\}^{\star}
               \{|\tilde{u}_2|^{r_2}, |\tilde{w}_2|^{r_2}\}^{\star}\, dx ,\\
        & \geq \int_{\R^N} |\tilde{u}_1|^{r_1} |\tilde{u}_2|^{r_2}
              + |\tilde{w}_1|^{r_1} |\tilde{w}_2|^{r_2} \, dx \\
        & = 
            \int_{\R^N} {(|u_1|^{r_1})}^{\tilde{}} {(|u_2|^{r_2})}^{\tilde{}}
              + {(|w_1|^{r_1})}^{\tilde{}} {(|w_2|^{r_2})}^{\tilde{}} \, dx,
            \\
        & \geq \int_{\R^N} |u_1|^{r_1} |u_2|^{r_2}
              + |w_1|^{r_1} |w_2|^{r_2} \, dx,
\end{align*}
and thus
\beq [re1]
J(u_1, u_2) + J(w_1, w_2) > J(\{\tilde{u}_1, \tilde{w}_1\}^{\star}, \{\tilde{u}_2, \tilde{w}_2\}^{\star}).
\eeq
Also from Lemma \ref{lem:relem} $(iii)$, for $i=1,2$,
\beq [re2]
\int_{\R^N} |\{\tilde{u}_i, \tilde{w}_i\}^{\star}|^2 \, dx
= \int_{\R^N} |\tilde{u}_i|^2 + |\tilde{w}_i|^2 \, dx
= \int_{\R^N} |u_i|^2 + |w_i|^2 \, dx,
\eeq
and taking \eqref{eq:3.16}-\eqref{eq:re2} and Lemma \ref{lem:m} $(iii)$ into consideration, one obtains the contradiction
$$m(a_1, a_2) > m(b_1, b_2) + m(a_1-b_1, a_2-b_2) \geq m(a_1, a_2).$$

\noindent{\it Case 2: $u_1 \neq 0$, $w_1 =0$ and $w_2 \neq 0$.} \\
If $u_2 \neq 0$, we can reverse the role of $u_1, w_1$ and $u_2, w_2$ in {\it Case 1} to get a contradiction. Thus, we suppose that
$u_2 = 0.$ Due to Lemma \ref{lem:relem} $(ii)$-$(v)$,

\beq [re3]
\begin{split}
J(\{\tilde{u}_1, 0\}^{\star}, \{\tilde{w}_2, 0\}^{\star})
&\leq  \frac12 \int_{\R^N} |\nabla \tilde{u}_1|^2 +  |\nabla \tilde{w}_2|^2 \, dx
        - \frac{\mu_1}{p_1} \int_{\R^N} |\tilde{u}_1|^{p_1} \, dx\\
&- \frac{\mu_2}{p_2} \int_{\R^N} |\tilde{w}_2|^{p_2} \, dx
        -\be \int_{\R^N} |\tilde{u}_1|^{r_1} |\tilde{w}_2|^{r_2}\\
& <J(\tilde{u}_1, 0) + J(0, \tilde{w}_2)\\
& \leq J(u_1, 0) + J(0, w_2),
\end{split}
\eeq

and
\beq [re4]
\begin{split}
\int_{\R^N} |\{\tilde{u}_1, 0\}^{\star}|^2 \, dx = \int_{\R^N} |\tilde{u}_1|^2 \, dx = \int_{\R^N} |u_1|^2 \, dx,\\
\int_{\R^N} |\{\tilde{w}_2, 0\}^{\star}|^2 \, dx = \int_{\R^N} |\tilde{w}_2|^2 \, dx = \int_{\R^N} |w_2|^2 \, dx.
\end{split}
\eeq
Thus using \eqref{eq:3.16}, \eqref{eq:re3}, \eqref{eq:re4} and Lemma \ref{lem:m}, we also have that
$$m(a_1, a_2) > m(b_1, b_2) + m(a_1-b_1, a_2-b_2) \geq m(a_1, a_2).$$
The contradictions obtained in {\it Cases} 1 and 2 indicate that $w_i^n(x) = u_i^n(x) - u_i(x+y_n) \rightarrow 0$ in $L^p(\R^N)$ for $2 < p < 2^*, i = 1, 2.$
\end{proof}

\begin{proof}[Proof of Theorem \ref{th:min}]
Let $\{(u_1^n, u_2^n)\}$ be a minimizing sequence for the functional $J$ on $S(a_1, a_2).$
In light of Lemma \ref{lem:conv}, we know that there exists $\{y_n\} \subset \R^N$ such that $u_i^n(x-y_n) \rightarrow u_i$ in $L^p(\R^N)$ for $2 < p < 2^*, i = 1, 2.$
Hence by weak convergence
\beq[min]
J(u_1, u_2) \leq m(a_1, a_2).
\eeq
Note that if $||u_1||_2^2 = a_1$ and $||u_2||_2^2 = a_2$ we are done. Indeed the strong convergence of $\{(u_1^n(\cdot -y_n), u_2^n(\cdot -y_n))\}$ in $H^1(\R^N) \times H^1(\R^N)$ then directly follows. To show that $||u_1||_2^2 = a_1$ and $||u_2||_2^2 = a_2$ we assume by contradiction that
$\|u_1\|_2^2 := b_1 < a_1$ or
$\|u_2\|_2^2 := b_2 < a_2.$ By definition $J(u_1,u_2) \geq m(b_1,b_2)$ and thus it results from \eqref{eq:min} that $m(b_1,b_2) \leq m(a_1,a_2)$. At this point since from Lemma \ref{lem:m} (iii) $m(a_1,a_2) \leq m(b_1,b_2) + m(a_1-b_1, a_2 - b_2)$ and by Lemma \ref{lem:m} (i) $m(a_1-b_1, a_2 - b_2) <0$ we have reached a contradiction and Theorem \ref{th:min} is proved.
\end{proof}

\begin{remark}\label{rem:radial}
As indicated in Remark \ref{rem:oneminimizer} a proof for the existence of minimizer for \eqref{eq:1.4} can be given working directly in $H_r^1(\R^N) \times H_r^1(\R^N)$. In such space the strong convergence in $L^p(\R^N) \times L^p(\R^N)$ is given for free. Now defining
\beq[1.4r]
m_r(a_1, a_2):=\inf_{(u_1, u_2)\in S_r(a_1, a_2)} J(u_1, u_2),
\eeq
where
$$
S_r(a_1, a_2): = \{(u_1, u_2)\in H_r^1(\R^N) \times H_r^1(\R^N): \|u_1\|_2^2 = a_1, \|u_2\|_2^2 = a_2\}
$$
we observe that we still have
\beq[1.5r]
m_r(a_1, a_2) \leq m_r(b_1, b_2)+ m_r(a_1-b_1, a_2-b_2),
\eeq
where $0\leq b_i \leq a_i$ for $ i=1, 2$. Indeed since we can choose a minimizing sequence which consists of Schwarz symmetric functions (which are in particular radially symmetric) it results that $m_r(c_1,c_2) = m(c_1,c_2)$ for any $c_1 \geq 0, c_2 \geq 0$ and \eqref{eq:1.5r} follows from Lemma \ref{lem:m} (iii). Thus we can end the proof as previously.
\end{remark}

\begin{remark}
In \cite{BJ}, \eqref{eq:1.5r} was not observed and  the fact that the weak limit belongs to $S_r(a_1,a_2)$ was proved using Liouville's type arguments, as developed in \cite{Ik}, see also \cite{CaChWe, Ik1}. It is the use of these arguments, which induces the restriction on the dimension $N$ in \cite[Theorem 2.1]{BJ}.
\end{remark}

\section{Proof of Theorem \ref{th:sta}}\label{sec:pf-th:sta}

Since our proof relies on the classical arguments of  
\cite{CL}, we only give a sketch.
\begin{proof}[Proof of Theorem \ref{th:sta}]
By contradiction, we assume that there is a $\eps_0 > 0$, $\{(\psi_1^n(0), \psi_2^n(0))\}$ and $\{t_n\} \subset \R^+$ such that
\begin{equation*}
\inf_{(u_1, u_2)\in G(a_1, a_2)}
\|(\psi_1^n(0), \psi_2^n(0)) - (u_1,u_2)\| \rightarrow 0,
\end{equation*}
and
\beq[4.1]
\inf_{(u_1, u_2) \in G(a_1, a_2)}
\|(\psi_1^n(t_n),\psi_2^n(t_n)) - (u_1, u_2)\| \geq \eps_0.
\eeq
Since by the conservation laws, 
$$
\|\psi_i^n(t_n)\|_2^2 = \|\psi_i^n(0)\|_2^2,
\quad J(\psi_1^n(t_n), \psi_2^n(t_n)) = J(\psi_1^n(0), \psi_2^n(0)), \quad \mbox{for } i=1,2,
$$
if we define
$$
\hat{\psi}_i^n = \frac{\psi_i^n(t_n)}{\|\psi_i^n(t_n)\|_2^2} a_i^{\frac12}, \quad \mbox{for } i=1,2,
$$
we get that
$$
\|\hat{\psi}_i^n\|_2^2 = a_i, \quad J(\hat{\psi}_1^n, \hat{\psi}_2^n) = m(a_1, a_2) + o(1).
$$
Namely $\{(\hat{\psi}_1^n, \hat{\psi}_2^n)\}$ is a minimizing sequence for \eqref{eq:1.4}. 
From Theorem \ref{th:min} it follows that it is precompact in that $H^1(\R^N) \times H^1(\R^N)$ thus \eqref{eq:4.1} fails.
\end{proof}

{\sc Address of the authors:}\\[1em]
\begin{tabular}{ll}
Tianxiang Gou & Louis Jeanjean\\
Laboratoire de Math\'ematiques (UMR 6623) & Laboratoire de Math\'ematiques (UMR 6623)\\
Universit\'{e} de Franche-Comt\'{e} & Universit\'{e} de Franche-Comt\'{e}\\
16, Route de Gray & 16, Route de Gray\\
25030 Besan\c{c}on Cedex & 25030 Besan\c{c}on Cedex\\
France & France\\
School of Mathematics and Statistics & louis.jeanjean@univ-fcomte.fr\\
Lanzhou University, Lanzhou, Gansu 730000\\
People's Republic of China\\
tianxiang.gou@univ-fcomte.fr
\end{tabular}

\end{document}